\documentclass[10pt]{article}
\usepackage{amsmath, amsthm, amssymb}
\usepackage{verbatim}
\usepackage[top=1.0in, bottom=1.0in, left=1.0in, right=1.0in]{geometry}
\theoremstyle{plain}
\newtheorem{thm}{Theorem}[section]
\newtheorem*{KernelMagic}{Kernel Magic}
\newtheorem*{MainTheorem}{Main Theorem}

\newtheorem{lem}[thm]{Lemma}

\newtheorem{cor}[thm]{Corollary}

\theoremstyle{definition}
\newtheorem{defn}{Definition}
\newtheorem*{TheDefinition}{Definition}
\theoremstyle{remark}

\newtheorem*{observation}{Observation}

\newcommand{\fancy}[1]{\mathcal{#1}}

\newcommand{\IN}{\mathbb{N}}
\newcommand{\IR}{\mathbb{R}}

\newcommand{\D}{\fancy{D}}
\newcommand{\T}{\fancy{T}}
\newcommand{\B}{\fancy{B}}
\renewcommand{\L}{\fancy{L}}
\newcommand{\HH}{\fancy{H}}

\newcommand{\set}[1]{\left\{ #1 \right\}}

\newcommand{\card}[1]{\left|#1\right|}
\newcommand{\size}[1]{\left\Vert#1\right\Vert}

\newcommand{\func}[3]{#1\colon #2 \rightarrow #3}

\newcommand{\irange}[1]{\left[#1\right]}

\newcommand{\parens}[1]{\left( #1 \right)}

\newcommand{\DefinedAs}{\mathrel{\mathop:}=}

\newcommand{\mic}{\operatorname{mic}}

\newcommand\restr[2]{{
  \left.\kern-\nulldelimiterspace 
  #1 
  \vphantom{\big|} 
  \right|_{#2} 
  }}

\def\D{\fancy{D}}

\newcommand{\claim}[2]{{\bf Claim #1.}~{\it #2}~~}

\newcommand\numberthis{\addtocounter{equation}{1}\tag{\theequation}}

\title{A better lower bound on average degree of $k$-list-critical graphs}
\author{Landon Rabern}
\begin{document}
\maketitle

\begin{abstract}
		We improve the best known bounds on average degree of $k$-list-critical graphs for $k \ge 6$. 
		Specifically, for $k \ge 7$ we show that every non-complete $k$-list-critical graph has average degree at least $k-1 + \frac{(k-3)^2 (2 k-3)}{k^4-2 k^3-11 k^2+28 k-14}$
		and every non-complete $6$-list-critical graph has average degree at least $5 + \frac{93}{766}$.
		The same bounds hold for online $k$-list-critical graphs.
\end{abstract}

\section{Introduction}
A graph $G$ is \emph{$k$-list-critical} if $G$ is not $(k-1)$-choosable, but every
proper subgraph of $G$ is $(k-1)$-choosable.  For further definitions and notation, see \cite{OreVizing, DischargingLowerBound}. 
Table \ref{TheTable} shows some history of lower bounds on average degree of $k$-list-critical graphs.

\begin{MainTheorem}
For $k \ge 7$, every non-complete $k$-list-critical graph has average degree at least \[k-1 + \frac{(k-3)^2 (2 k-3)}{k^4-2 k^3-11 k^2+28 k-14}.\]
Every non-complete $6$-list-critical graph has average degree at least $5 + \frac{93}{766}$.
\end{MainTheorem}

The proof is similar to the $4$-list-critical case in \cite{Better4ListCriticalBound}, but now we incorporate reducibility lemmas from Kierstead and R. \cite{OreVizing}.  
Basically, we show that the average degree of the subgraph induced on vertices of degree $k-1$ is small, which implies that the number of edges incident to the vertices of degree at least $k$ must be large,
and hence the number of vertices of degree at least $k$ must be large; that is, the graph must have high average degree.  That is how all known proofs of lower bounds on average degree of $k$-list-critical graphs work.
A tight bound on the average degree of the subgraph induced on vertices of degree $k-1$ in a $k$-list-critical graph was proved by Gallai \cite{gallai1963kritische}.  The connected graphs in which each block is a complete graph
or an odd cycle are called \emph{Gallai trees}.  Gallai \cite{gallai1963kritische} proved that in a $k$-critical graph, the vertices of degree $k-1$ induce a disjoint union of Gallai trees.  
The same is true for $k$-list-critical graphs \cite{borodin1977criterion, erdos1979choosability}.  Since Gallai's bound is tight, it may appear that there is no hope of improvement using the above method.  
While it is true that the upper bound on average degree of Gallai trees cannot be improved in general, it can be improved in the absence of certain \emph{bad} properties.  
Let $G$ be a $k$-list-critical graph and let $\L$ be the subgraph of $G$ induced on vertices of degree $k-1$.
If the presence of bad properties in $\L$ could be shown to lead to reducible configurations in $G$, we would have a pathway to improvement.  Kostochka and Stiebitz \cite{kostochkastiebitzedgesincriticalgraph}
made the first progress along these lines.  Further improvements in \cite{OreVizing}, \cite{DischargingLowerBound} and \cite{Better4ListCriticalBound} follow the same general outline.  
As in \cite{DischargingLowerBound} and \cite{Better4ListCriticalBound}, it is convenient to have a measure of how bad $\L$ is.  So, if $b$ is a function measuring badness, this could be realized as
an upper bound of the form:
\[2\size{\L} \le s(k)\card{\L} + t(k)b(\L).\]
Of course, we can measure badness along multiple axes (in badness space?).  In our proof we use two badness measures $\beta(\L)$ and $q(\L)$, so the upper bound looks like:
\[2\size{\L} \le s(k)\card{L} + h(k)\beta(\L) + z(k)q(\L).\]
High $\beta(\L)$ badness leads to reducible configurations by kernel-perfect orientations and high $q(\L)$ badness leads to reducible configurations by Alon-Tarsi orientations.
That means the same proof shows that Main Theorem holds for online $k$-list-critical graphs as well (in fact, for the larger class of $OC$-irreducible graphs with $\delta(G) = k-1$ defined in section 5).

Let $c_k^*(\L)$ be the number of components of $\L$ containing a copy of $K_{k-1}$. Let $q_k(\L)$ be the number of non-cut vertices in $\L$ that appear in copies of $K_{k-1}$.  
Let $\beta_k(\L)$ be the independence number of the subgraph of $\L$ induced on the vertices of degree $k-1$.  
When $k$ is defined in context, we just write $c^*(\L)$, $q(\L)$ and $\beta(\L)$.  
The following upper bounds on $q(\L)$ and $\beta(\L)$ are likely to be reusable.  More general versions of these lemmas are stated and proved in sections 4 and 5.

\begin{lem}\label{qLemmaList}
	Let $G$ be a non-complete $k$-list-critical graph where $k \ge 5$.  Let $\L$ be the subgraph of $G$ induced on $(k-1)$-vertices, $\HH^-$ the subgraph of $G$ induced on $k$-vertices
	and $\HH^+$ the subgraph of $G$ induced on $(k+1)^+$-vertices.  Then
	\[q(\L) \le c^*(\L) + 4\card{\HH^-} + \size{\HH^+, \L},\] and if $k \ge 7$, then
	\[q(\L) \le 2c^*(\L) + 3\card{\HH^-} + \size{\HH^+, \L}.\]
\end{lem}

\begin{lem}\label{betaLemmaList}
	Let $G$ be a $k$-list-critical graph.  Let $\L$ be the subgraph of $G$ induced on $(k-1)$-vertices and	$\HH$ the subgraph of $G$ induced on $k^+$-vertices.  
	If $2 \le \lambda \le \frac{6(k-1)}{k}$, then
	\[\beta(\L) \le \frac{2}{\lambda}\size{\HH} + \frac{2\size{G} - (k-2)\card{G} - \parens{\frac{k}{2} + \frac{k-1}{\lambda}}\card{\HH} - 1}{k-1}.\]
\end{lem}

\begin{table}
	\begin{center}
		\begin{tabular}{|c|c|c|c|c|c|c|c|c|}
			\hline
			& Gallai \cite{gallai1963kritische}
			& KS \cite{kostochkastiebitzedgesincriticalgraph} 
			& KR \cite{OreVizing}
			& CR \cite{DischargingLowerBound}
			& R \cite{Better4ListCriticalBound}
			& Here \\
			$k$ & $d(G) \ge$ & $d(G) \ge$ & $d(G) \ge$ & $d(G) \ge$ & $d(G) \ge$ & $d(G) \ge$\\
			\hline 
			4 & 3.0769 & --- & --- & --- & \bf{3.1000} & 3.1000\\
			5 & 4.0909 & --- & 4.0984 & 4.1000 & \bf{4.1176} & 4.1176\\
			6 & 5.0909 & --- & 5.1053 & 5.1076 & 5.1153 & \bf{5.1214}\\
			7 & 6.0870 & --- & 6.1149 & 6.1192 & 6.1081 & \bf{6.1296}\\
			8 & 7.0820 & --- & 7.1128 & 7.1167 & 7.1000 & \bf{7.1260}\\
			9 & 8.0769 & 8.0838 & 8.1094 & 8.1130 & 8.0923 & \bf{8.1213}\\
			10 & 9.0722 & 9.0793 & 9.1055 & 9.1088 & 9.0853 & \bf{9.1162}\\
			15 & 14.0541 & 14.0610 & 14.0864 & 14.0884 & 14.0609 & \bf{14.0930}\\
			20 & 19.0428 & 19.0490 & 19.0719 & 19.0733 & 19.0469 & \bf{19.0762}\\
			\hline
		\end{tabular}
	\end{center}
	\caption{Lower bounds on average degree $d(G)$ of a $k$-list-critical graph $G$.}
	\label{TheTable}
\end{table}

\section{General lower bounds on average degree}
This is the counting portion of the proof, which is simpler and more general than the counting in \cite{OreVizing} and \cite{DischargingLowerBound}.
\begin{defn}
A quadruple $\parens{p,h,z,f}$ of functions from $\IN$ to $\IR$ is \emph{$r$-Gallai} if for every $k \ge r$ and Gallai tree $T \ne K_k$ with $\Delta(T) \le k-1$,
the following hold:
\begin{itemize}
\item if $K_{k-1} \subseteq T$, then $2\size{T} \le \parens{k-3 + p(k)}\card{T} + h(k)q(T) + z(k)\beta(T) + f(k)$; and
\item if $K_{k-1} \not\subseteq T$, then $2\size{T} \le \parens{k-3 + p(k)}\card{T} + z(k)\beta(T)$.
\end{itemize}
\end{defn}

\begin{thm}\label{k7}
	Let $\parens{p,h,z,f}$ be $7$-Gallai.  If $k \ge 7$ and $2 \le z(k) \le \frac{6(k-1)}{k}$, then for any non-complete $k$-list-critical graph $G$,
	\[d(G) \ge k-1 + \frac{2 - p(k) - \frac{z(k)}{k-1} + \frac{\frac{z(k)}{k-1} - (2h(k) + f(k))c^*(\L)}{\card{G}}}{k+1 + 3h(k) - p(k) - \frac{(k-2)z(k)}{2(k-1)}},\]
	where $\L$ is the subgraph of $G$ induced on $(k-1)$-vertices.
\end{thm}
\begin{proof}
Let $\HH^-$ the subgraph of $G$ induced on $k$-vertices, $\HH$ the subgraph of $G$ induced on $k^+$-vertices, 
$\HH^+$ the subgraph of $G$ induced on $(k+1)^+$-vertices and $\D$ the components of $\L$ containing $K_{k-1}$.
Plainly, the following bounds hold.
\begin{equation}\label{eq1c}
2\size{G} \ge k\card{G} - \card{\L}
\end{equation}
\begin{equation}\label{eq2c}
2\size{G} \ge (k+1)\card{G} - \card{\HH^-} - 2\card{\L}
\end{equation}
\begin{equation}\label{eq3c}
2\size{G} \ge k\card{\HH^-} + (k-1)\card{\L} + \size{\HH^+, \L}
\end{equation}
\begin{equation}\label{eq4c}
\size{\HH,\L} = (k-1)\card{\L} - 2\size{\L}
\end{equation}
Since $\parens{p,h,z,f}$ is $7$-Gallai,
\begin{equation}\label{eq5c}
2\size{\L} \le \parens{k-3 + p(k)}\card{\L} + f(k)\card{\D} + h(k)q(\L) + z(k)\beta(\L)
\end{equation}
By Lemma \ref{qLemmaList},
\[q(\L) \le 2\card{\D} + 3\card{\HH^-} + \size{\HH^+, \L},\]
plugging this into \eqref{eq5c} gives
\begin{equation}\label{eq6c}
2\size{\L} \le \parens{k-3 + p(k)}\card{\L} + 3h(k)\card{\HH^-} + h(k)\size{\HH^+, \L} + z(k)\beta(\L) + S_1,
\end{equation}
where
\[S_1 \DefinedAs (2h(k) + f(k))\card{\D}.\]
Now using \eqref{eq1c} and \eqref{eq6c},
\begin{align*}
	2\size{G} &= 2\size{\HH} + 2\size{\HH, \L} + 2\size{\L}\\
	&= 2\size{\HH} + 2((k-1)\card{\L} - 2\size{\L}) + 2\size{\L}\\
	&= 2\size{\HH} + 2(k-1)\card{\L} - 2\size{\L}\\
	&\ge 2\size{\HH} + \parens{k+1 - p(k)}\card{\L} - 3h(k)\card{\HH^-} - h(k)\size{\HH^+, \L} - z(k)\beta(\L) - S_1\numberthis \label{eq7c}\\
\end{align*}
Adding $h(k)$ times \eqref{eq3c} to \eqref{eq7c} gives
\begin{equation}\label{eq8c}
2\size{G} \ge \frac{2\size{\HH} + \parens{k+1 +(k-1)h(k)- p(k)}\card{\L} + (k- 3)h(k)\card{\HH^-}  - z(k)\beta(\L) - S_1}{1 + h(k)}
\end{equation}
Lemma \ref{betaLemmaList} gives
\[\beta(\L) \le \frac{2}{z(k)}\size{\HH} + \frac{2\size{G} - (k-2)\card{G} - \parens{\frac{k}{2} + \frac{k-1}{z(k)}}\card{\HH} - 1}{k-1}.\]
Plugging this into \eqref{eq8c} yields
\begin{equation}\label{eq9c}
2\size{G} \ge \frac{\parens{k+1 +(k-1)h(k)- p(k)}\card{\L} + (k- 3)h(k)\card{\HH^-} +\frac{(k-2)z(k)}{k-1}\card{G} + \parens{\frac{kz(k)}{2(k-1)} + 1}\card{\HH} + S_2}{1 + h(k) + \frac{z(k)}{k-1}},
\end{equation}
where
\[S_2 \DefinedAs \frac{z(k)}{k-1} - S_1.\]
Now using $\card{\HH} = \card{G} - \card{\L}$ gives
\begin{equation}\label{eq10c}
2\size{G} \ge \frac{\parens{k +(k-1)h(k)- p(k) - \frac{kz(k)}{2(k-1)}}\card{\L} + (k- 3)h(k)\card{\HH^-} + \parens{\frac{(3k-4)z(k)}{2(k-1)} + 1}\card{G} + S_2}{1 + h(k) + \frac{z(k)}{k-1}}.
\end{equation}
Now using \eqref{eq2c} to get a lower bound on $\card{\HH^-}$ gives
\begin{equation}\label{eq11c}
2\size{G} \ge \frac{\parens{k - (k-5)h(k)- p(k) - \frac{kz(k)}{2(k-1)}}\card{\L}+\parens{(k+1)(k-3)h(k) + \frac{(3k-4)z(k)}{2(k-1)} + 1}\card{G} + S_2}{1 + (k-2)h(k) + \frac{z(k)}{k-1}}.
\end{equation}
Using \eqref{eq1c} to get a lower bound on $\card{\L}$ and simplifying gives
\begin{equation}\label{eq12c}
\frac{2\size{G}}{\card{G}} \ge \frac{k^2 + 3(k-1)h(k) - kp(k) + 1 - \frac{k^2-3k+4}{2(k-1)}z(k) + \frac{S_2}{\card{G}}}{k+1 + 3h(k) - p(k) - \frac{(k-2)z(k)}{2(k-1)}}.
\end{equation}
Now factoring out $k-1$ gives the desired bound.
\end{proof}

\noindent A nearly identical argument, using the other inequality in Lemma \ref{qLemmaList}, proves a bound that holds for $k \ge 5$.

\begin{thm}\label{k5}
	Let $\parens{p,h,z,f}$ be $5$-Gallai.  If $k \ge 5$ and $2 \le z(k) \le \frac{6(k-1)}{k}$, then for any non-complete $k$-list-critical graph $G$,
	\[d(G) \ge k-1 + \frac{2 - p(k) - \frac{z(k)}{k-1} + \frac{\frac{z(k)}{k-1} - (h(k) + f(k))c^*(\L)}{\card{G}}}{k+1 + 4h(k) - p(k) - \frac{(k-2)z(k)}{2(k-1)}},\]
	where $\L$ is the subgraph of $G$ induced on $(k-1)$-vertices.
\end{thm}

\noindent When $k=4$, we cannot apply Lemma \ref{qLemmaList}, but using $h(k)=0$ and running through the same argument proves the following bound for $k\ge 4$.
\begin{thm}\label{k4}
	Let $\parens{p,0,z,f}$ be $4$-Gallai.  If $k \ge 4$ and $2 \le z(k) \le \frac{6(k-1)}{k}$, then for any non-complete $k$-list-critical graph $G$,
	\[d(G) \ge k-1 + \frac{2 - p(k) - \frac{z(k)}{k-1} + \frac{\frac{z(k)}{k-1} - f(k)c^*(\L)}{\card{G}}}{k+1 - p(k) - \frac{(k-2)z(k)}{2(k-1)}},\]
	where $\L$ is the subgraph of $G$ induced on $(k-1)$-vertices.
\end{thm}

When $z(k) < 2$, using Lemma \ref{betaLemmaList} worsens the lower bound, so we may as well use $z(k)=0$; that is, drop the $\beta(\L)$ term entirely.  
Doing so in the above argument shows that Theorems \ref{k7}, \ref{k5}, \ref{k4} hold for $z(k) = 0$ if we replace $k+1$ in the denominator with $k+2$.  
This gives the bounds proved by discharging in Cranston and R. \cite{DischargingLowerBound}.

\section{Gallai quadruples}
All known proofs of lower bounds for average degree of list-critical graphs are essentially a counting argument combined with the fact that some quadruple is Gallai.

\begin{lem}[Gallai \cite{gallai1963kritische}]
$\parens{\frac{k+1}{k-1}, 0, 0, -2}$ is $4$-Gallai.
\end{lem}

\begin{lem}[Kostochka-Stiebitz \cite{kostochkastiebitzedgesincriticalgraph}]
$\parens{\frac{4(k-1)}{k^2 - 3k + 4}, \frac{k^2 - 3k}{k^2-3k+4}, 0, \frac{-4(k^2-3k+2)}{k^2-3k+4}}$ is $7$-Gallai.
\end{lem}

\begin{lem}[Cranston-R. \cite{DischargingLowerBound}]
$\parens{\frac{3k-5}{k^2-4k+5}, \frac{k(k-3)}{k^2-4k+5}, 0, \frac{-2(k-1)(2k-5)}{k^2-4k+5}}$ is $5$-Gallai.
\end{lem}

\begin{lem}[R. \cite{Better4ListCriticalBound}]\label{Rbound}
$\parens{1, 0, 2, 0}$ is $4$-Gallai.
\end{lem}

We give a a list of inequalities that provide a sufficient condition for $(p,h,z,f)$ to be $5$-Gallai.  
These inequalities take a form quite similar to the inequalities in Cranston and R. \cite{DischargingLowerBound}, but now
they involve $z(k)$ as well.   The sufficiency proof is a small modification of the proof in \cite{DischargingLowerBound}. 
To use a Gallai quadruple in Lemma \ref{k7}, we want $2h(k) + f(k) \le 0$ to get rid of the term involving $c^*(\L)$.  Similarly,
for Lemma \ref{k5}, we want $h(k) + f(k) \le 0$.   Finding the $p,h,z,f$ that give the largest average degree subject to these constraints
is a fractional linear program that can be converted to a linear program and solved for each $k$.  This is useful
for verification of bounds, but we want a formula in terms of $k$.  For $k \ge 7$, we use the following quadruple.

\begin{lem}\label{Gallai7Up}
$\parens{\frac{3k-7}{k^2-4k+5}, \frac{(k-1)(k-4)}{k^2-4k+5}, 2, \frac{-2(k-1)(k-4)}{k^2-4k+5}}$ is $5$-Gallai.
\end{lem}

\noindent For $k=6$, we use the following quadruple.  For $k=5$, the quadruple in Lemma \ref{Rbound} is the optimal choice of $p,h,z,f$.

\begin{lem}\label{Gallai6Up}
$\parens{\frac{3k-5}{k^2-3k+3}, \frac{(k-1)(k-4)}{k^2-3k+3}, \frac{(3k-5)(k-2)}{k^2-3k+3}, \frac{-(k-1)(k-4)}{k^2-3k+3}}$ is $5$-Gallai.
\end{lem}

\noindent Now on to the sufficiency proof.  For an endblock $B$ of a Gallai tree $T$, let $x_B$ be the cutvertex contained in $B$.

\begin{lem}\label{nokkm1}
	Let $\func{z}{\IN}{\IR}$ such that $z(k) = 0$ or $z(k) \ge 2$ for all $k \in \IN$.
	For all $k \ge 5$ and Gallai trees $T$ with $\Delta(T) \le k-1$ and $K_{k-1} \not \subseteq T$, we have
	\[2\size{T} \le \parens{k-3 + \frac{\max\set{2, 3-z(k)}}{k-2}}\card{T} + z(k)\beta(T).\]
\end{lem}
\begin{proof}
Suppose the lemma is false and choose a counterexample $T$ minimizing $\card{T}$.  

\noindent\claim{1}{$T$ has at least two blocks.}

If $T$ has only one block, then $2\size{T} \le \parens{k-3}\card{T}$.

\noindent\claim{2}{Each endblock of $T$ is $K_{k-2}$.}

Suppose $T$ has an endblock $B$ that is not $K_{k-2}$. Then removing $V(B)\setminus\set{x_B}$ from $T$ to get $T'$ and applying minimality of $\card{T}$ gives
\[2\size{B} > \parens{k-3 + \frac{\max\set{2, 3-z(k)}}{k-2}}\parens{\card{B} - 1}.\]
This is a contradiction unless $k=5$ and $B = K_3$, but then $B = K_{k-2}$, a contradiction.

\noindent\claim{3}{If $B$ is an endblock of $T$, then $d_T(x_B) = k-1$.}

Suppose $B$ is an endblock of $T$ with $d_T(x_B) < k-1$.  Then $B = K_{k-2}$ by Claim 2 and hence $d_T(x_B) = k-2$.  Removing $V(B)$ from $T$ to get $T^*$ and applying minimality of $\card{T}$ gives
the contradiction
\[(k-2)(k-3) + 6 > \parens{k-3 + \frac{\max\set{2, 3-z(k)}}{k-2}}\parens{k-1}.\]

\noindent\claim{4}{$T$ does not exist.}

By the previous claims, we know that every endblock $T$ is a $K_{k-2}$ that shares a vertex with an odd cycle.  Pick and endblock $B$ that is the end of a
longest path in the block-tree of $T$.  Let $C$ be the odd cycle sharing $x_B$ with $B$.  Since $B$ is the end of a longest path in the block-tree, there is
a neighbor $y$ of $x_B$ on $C$ such that $d_T(y) = 2$ or $y$ is contained in another endblock $A$ (which must be a $K_{k-2}$).  
First, suppose $d_T(y) = 2$.
Removing $V(B) \cup \set{y}$ from $T$ to get $T'$ and applying minimality of $\card{T}$ gives the contradiction (since $\beta(T') < \beta(T)$)
\[(k-2)(k-3) + 6 > \parens{k-3 + \frac{\max\set{2, 3-z(k)}}{k-2}}\parens{k-1} + z(k)(\beta(T) - \beta(T')).\]
Hence $y$ is contained in another $K_{k-2}$ endblock $A$.  Removing $V(B) \cup V(A)$ from $T$ to get $T^*$ and applying minimality of $\card{T}$ gives
the contradiction (since $\beta(T^*) < \beta(T)$)
\[2(k-2)(k-3) + 6 > \parens{k-3 + \frac{\max\set{2, 3-z(k)}}{k-2}}\parens{2(k-2)} + z(k)(\beta(T) - \beta(T^*)).\]
\end{proof}

\begin{lem}\label{bwkk1}
	Let $\func{p}{\IN}{\IR_{\ge 0}}$, $\func{f}{\IN}{\IR}$, $\func{h}{\IN}{\IR_{\ge 0}}$, $\func{z}{\IN}{\IR_{\ge 0}}$ such that $z(k) = 0$ or $z(k) \ge 2$.
	For all $k \ge 5$ and Gallai trees $T \ne K_k$ with $\Delta(T) \le k-1$ and $K_{k-1} \subseteq T$, we have
	\[2\size{T} \le (k-3 + p(k))\card{T} + f(k) + h(k)q(T) + z(k)\beta(T)\]
	whenever $p$, $f$, $h$ and $z$ satisfy all of the following conditions:
	\begin{enumerate}
		\item[(1)] $f(k) \ge (k-1)(1- p(k) - h(k))$; and	
	    \item[(2)] $p(k) \ge \frac{3 - \frac{z(k)}{2}}{k-2}$; and
		\item[(3)] $p(k) \ge h(k) + 5 - k$; and
		\item[(4)] $p(k) \ge \frac{2+h(k)}{k-2}$; and
		\item[(5)] $(k-1)p(k) + (k-3)h(k) + z(k)\ge k+1$.
	\end{enumerate}
\end{lem}
\begin{proof}
Suppose the lemma is false and choose a counterexample $T$ minimizing $\card{T}$.  

\noindent\claim{1}{$T$ has at least two blocks.}

Otherwise, $T = K_{k-1}$ and (1) gives a contradiction.

\noindent\claim{2}{Each endblock of $T$ is $K_{k-2}$ or $K_{k-1}$.}

Suppose $T$ has an endblock $B$ that is not $K_{k-2}$ or $K_{k-1}$. Then removing $V(B)\setminus\set{x_B}$ from $T$ to get $T'$ and applying minimality of $\card{T}$ gives
\[2\size{B} > \parens{k-3 + p(k)}\parens{\card{B} - 1} + h(k)(q(T) - q(T')) + z(k)(\beta(T) - \beta(T')).\]
If $B = K_2$, then $q(T') \le q(T) + 1$, otherwise $q(T') = q(T)$.  For $B=K_2$, we have to contradiction (to (3))
\[2 > \parens{k-3 + p(k)} - h(k).\]
Suppose $B = K_t$ for $4 \le t \le k-3$.  Then we have the contradiction
\[t(t-1) > \parens{k-3 + p(k)}\parens{t-1}.\]
Finally, suppose $B$ is an odd cycle of length $\ell$.  Then, we have
\[2\ell > \parens{k-3 + p(k)}\parens{\ell-1}.\]
This simplifies to
\[\ell < 1 + \frac{2}{k-5+p(k)}.\]
Since $k-5+p(k) \ge 1$ when $k \ge 6$, this implies that $k = 5$.  Using (4), we conclude $\ell = 3$, but then $B = K_{k-2}$, a contradiction.

\noindent\claim{3}{$T$ has at most one $K_{k-1}$ endblock.}

Suppose $T$ has at least two $K_{k-1}$ endblocks.  Let $B$ be one of them.    Then removing $V(B)$ from $T$ leaves a graph $T'$ with $K_{k-1} \subseteq T'$.  So, we may apply minimality of $\card{T}$ to get
\[(k-1)(k-2) + 2 > \parens{k-3 + p(k)}\parens{k-1} + h(k)(q(T) - q(T')) + z(k)(\beta(T) - \beta(T')).\]
Now $\beta(T') < \beta(T)$ and $q(T') \le q(T) - (k-2) + 1$, so we have the contradiction (to (5))
\[k+1 > (k-1)p(k) + (k-3)h(k) + z(k).\]

\noindent\claim{4}{If $B$ is an endblock of $T$, then $d_T(x_B) = k-1$.}

Suppose $B$ is an endblock of $T$ with $d_T(x_B) < k-1$.  Then $B = K_{k-2}$ by Claim 2.  Removing $V(B)$ from $T$ leaves a graph $T'$ with $K_{k-1} \subseteq T'$.  So, we may apply minimality of $\card{T}$ to get
\[(k-2)(k-3) + 2 > \parens{k-3 + p(k)}\parens{k-2} + h(k)(q(T) - q(T')) + z(k)(\beta(T) - \beta(T')).\]
We have $q(T') \le q(T) + 1$, so this is gives the contradiction (to (4))
\[2 > (k-1)p(k) - h(k).\]

\noindent\claim{5}{$T$ does not exist.}

By Claims 2 and 3, all but at most one endblock of $T$ is $K_{k-2}$ with a cutvertex that is also in an odd cycle. 
Pick and endblock $B$ that is the end of a
longest path in the block-tree of $T$.  Let $C$ be the odd cycle sharing $x_B$ with $B$.  Since $B$ is the end of a longest path in the block-tree, there is
a neighbor $y$ of $x_B$ on $C$ such that $d_T(y) = 2$ or $y$ is contained in another endblock $A$ (which must be a $K_{k-2}$).  
First, suppose $d_T(y) = 2$.
Removing $V(B) \cup \set{y}$ from $T$ to get $T'$ and applying minimality of $\card{T}$ gives (since $q(T') = q(T)$ and $\beta(T') < \beta(T)$)
\[(k-2)(k-3) + 6 > \parens{k-3 + p(k)}\parens{k-1} + z(k),\]
so
\[p(k) < \frac{9-k-z(k)}{k-1},\]
contradicting (2).
Hence $y$ is contained in another $K_{k-2}$ endblock $A$.  Removing $V(B) \cup V(A)$ from $T$ to get $T^*$ and applying minimality of $\card{T}$ gives(since $q(T^*) = q(T)$ and $\beta(T^*) < \beta(T)$)
\[2(k-2)(k-3) + 6 > \parens{k-3 + p(k)}\parens{2(k-2)} + z(k),\]
so
\[6 > 2(k-2)p(k) + z(k),\]
contradicting (2).
\end{proof}

The proof of Lemma \ref{Gallai7Up} and Lemma \ref{Gallai6Up} are now straightforward computations.  
That is all we need to prove our lower bounds on average degree. If a good upper bound on $c^*(\L)$ is known, it may be better to allow $2h(k) + f(k) > 0$.
In that case, one could use the following.

\begin{lem}
If $\func{z}{\IN}{\IR}$ is such that $z(k) = 0$ or $2 \le z(k) \le \frac{k(k-3)}{k-2}$  for all $k \in \IN$, then 
$(p,h,z,f)$ is $5$-Gallai, where
\[h(k) \DefinedAs \frac{k(k-3) - (k-2)z(k)}{k^2-4k+5},\]
\[p(k) \DefinedAs \frac{2 + h(k)}{k-2},\]
\[f(k) \DefinedAs (k-1)(1 - h(k) - p(k)).\]
\end{lem}
\section{Bounding $q(\L)$}
This section is devoted to extracting the reusable Lemma \ref{qLemma} from the proof of Kierstead and R. \cite{OreVizing}.
All of the hard work was already done in \cite{OreVizing}.

\begin{defn}
	A graph $G$ is \emph{AT-reducible} to $H$ if $H$ is a nonempty induced subgraph of $G$ which is $f_H$-AT where $f_H(v) \DefinedAs \delta(G) + d_H(v) - d_G(v)$ for all $v \in V(H)$.  
	If $G$ is not AT-reducible to any nonempty induced subgraph, then it is \emph{AT-irreducible}.
\end{defn}

\begin{lem}\label{qLemma}
	Let $G$ be a non-complete AT-irreducible graph with $\delta(G) = k-1$ where $k \ge 5$.  Let $\L$ be the subgraph of $G$ induced on $(k-1)$-vertices, $\HH^-$ the subgraph of $G$ induced on $k$-vertices and 
	$\HH^+$ the subgraph of $G$ induced on $(k+1)^+$-vertices.  Then
	\[q(\L) \le c^*(\L) + 4\card{\HH^-} + \size{\HH^+, \L},\] and if $k \ge 7$, then
	\[q(\L) \le 2c^*(\L) + 3\card{\HH^-} + \size{\HH^+, \L}.\]
\end{lem}

\begin{observation}
The hypotheses of Lemma \ref{qLemma} are satisfied by non-complete $k$-critical, $k$-list-critical, online $k$-list-critical and $k$-AT-critical graphs.
\end{observation}

The proof of Lemma \ref{qLemma} requires the following four lemmas from \cite{OreVizing}.

\begin{lem}\label{DegenerateEuler}
Let $G$ be a graph and $\func{f}{V(G)}{\IN}$.  If $\size{G} > \sum_{v \in V(G)} f(v)$, then $G$ has an induced subgraph $H$ such that $d_H(v) > f(v)$ for each $v \in V(H)$.
\end{lem}
\begin{proof}
Suppose not and choose a counterexample $G$ minimizing $\card{G}$. Then $\card{G} \ge 3$ and we have $x \in V(G)$ with $d_G(x) \leq f(x)$. But now $\size{G-x} > \sum_{v \in V(G-x)} f(v)$, contradicting minimality of $\card{G}$.
\end{proof}

Let $\T_k$ be the Gallai trees with maximum degree at most $k-1$, excepting $K_k$. For a graph $G$, let $W^k(G)$ be the set of vertices of $G$ that are contained in some $K_{k-1}$ in $G$.  

\begin{lem}\label{ConfigurationTypeOneEuler}
Let $k \ge 5$ and let $G$ be a graph with $x \in V(G)$ such that:
\begin{enumerate}
\item $K_k \not \subseteq G$; and
\item $G-x$ has $t$ components $H_1, H_2, \ldots, H_t$, and all are in $\T_k$; and
\item $d_G(v) \leq k - 1$ for all $v \in V(G-x)$; and
\item $\card{N(x) \cap W^k(H_i)} \ge 1$ for $i \in \irange{t}$; and
\item $d_G(x) \ge t+2$.
\end{enumerate}

\noindent Then $G$ is $f$-AT where $f(x) = d_G(x) - 1$ and $f(v) = d_G(v)$ for all $v \in V(G - x)$.
\end{lem}

For a graph $G$, $\set{X, Y}$ a partition of $V(G)$ and $k \ge 4$, let $\B_k(X, Y)$ be the bipartite graph with one part $Y$ and the other part the components of $G[X]$.  Put an edge between $y \in Y$ and a component $T$ of $G[X]$ iff $N(y) \cap W^k(T) \ne \emptyset$.   The next lemma tells us that we have a reducible configuration if this bipartite graph has minimum degree at least three.  

\begin{lem}
	\label{MultipleHighConfigurationEuler} Let $k\ge7$ and let $G$ be a graph with
	$Y\subseteq V(G)$ such that: 
	\begin{enumerate}
		\item $K_{k}\not\subseteq G$; and 
		\item the components of $G-Y$ are in $\T_{k}$; and 
		\item $d_{G}(v)\leq k-1$ for all $v\in V(G-Y)$; and 
		\item with $\B\DefinedAs\B_{k}(V(G-Y),Y)$ we have $\delta(\B)\ge3$. 
	\end{enumerate}
	\noindent Then $G$ has an induced subgraph $G'$ that is $f$-AT where $f(y)=d_{G'}(y)-1$
	for $y\in Y$ and $f(v)=d_{G'}(v)$ for all $v\in V(G'-Y)$.\end{lem}

We also have the following version with asymmetric degree condition on $\B$.  
The point here is that this works for $k \ge 5$.  
The consequence is that we trade a bit in our bound for the proof to go through with $k \in \set{5,6}$.

\begin{lem}
	\label{MultipleHighConfigurationEulerLopsided} Let $k \ge 5$ and let $G$ be a graph with
	$Y\subseteq V(G)$ such that: 
	\begin{enumerate}
		\item $K_{k}\not\subseteq G$; and 
		\item the components of $G-Y$ are in $\T_{k}$; and 
		\item $d_{G}(v)\leq k-1$ for all $v\in V(G-Y)$; and 
		\item with $\B \DefinedAs \B_k(V(G-Y), Y)$ we have $d_{\B}(y) \ge 4$ for all $y \in Y$ and $d_{\B}(T) \ge 2$ for all components $T$ of $G-Y$.
	\end{enumerate}
	\noindent Then $G$ has an induced subgraph $G'$ that is $f$-AT where $f(y)=d_{G'}(y)-1$
	for $y\in Y$ and $f(v)=d_{G'}(v)$ for all $v\in V(G'-Y)$.
\end{lem}

\begin{proof}[Proof of Lemma \ref{qLemma}]
Let $\HH$ be the subgraph of $G$ induced on $k^+$-vertices and let $\D$ be the components of $\L$ containing a copy of $K_{k-1}$. 
Put $W \DefinedAs W^k(\L)$ and $L' \DefinedAs V(\L) \setminus W$. Define an auxiliary bipartite graph $F$ with parts $A$ and $B$ where:
\begin{enumerate}
\item  $B = V(\HH^-)$ and $A$ is the disjoint union of the following sets
$A_1, A_2$ and $A_3$,
\item $A_1 = \D$ and each $T \in \D$ is adjacent to all $y \in B$
where $N(y) \cap W^k(T) \ne \emptyset$,
\item For each $v \in L'$, let $A_2(v)$ be a set of $\card{N(v) \cap
B}$ vertices connected to $N(v) \cap B$ by a matching in $F$.  Let
$A_2$ be the disjoint union of the $A_2(v)$ for $v \in L'$,
\item For each $y \in B$, let $A_3(y)$ be a set of $d_{\HH}(y)$ vertices
which are all joined to $y$ in $F$.  Let $A_3$ be the disjoint union
of the $A_3(y)$ for $y \in B$.
\end{enumerate}

Define $\func{f}{V(F)}{\IN}$ by $f(v) = 1$ for all $v \in A_1 \cup A_2 \cup A_3$ and $f(v) = 3$ for all $v \in B$.  First, suppose $\size{F} > \sum_{v \in V(F)} f(v)$.  
Then by Lemma \ref{DegenerateEuler}, $F$ has an induced subgraph $Q$ such that $d_Q(v) > f(v)$ for each $v \in V(Q)$.  
In particular, $V(Q) \subseteq B \cup A_1$ and $d_Q(v) \ge 4$ for $v \in B \cap V(Q)$ and $d_Q(v) \ge 2$ for $v \in A_1 \cap V(Q)$.  
Put $Y \DefinedAs B \cap V(Q)$ and let $X$ be $\bigcup_{T \in V(Q) \cap A_1} V(T)$. 
Now $Z \DefinedAs G[X \cup Y]$ satisfies the hypotheses of Lemma \ref{MultipleHighConfigurationEulerLopsided}, so $Z$ has an induced subgraph $G'$ that is $f$-AT 
where $f(y) = d_{G'}(y) - 1$ for $y \in Y$ and $f(v) = d_{G'}(v)$ for $v \in X$.  Since $Y \subseteq B$ and $X \subseteq V(\L)$, we have $f(v) = k-1 + d_{G'}(v) - d_G(v)$ for all $v \in V(G')$.  
Hence, $G$ is AT-reducible to $G'$, a contradiction.
Therefore $\size{F} \le \sum_{v \in V(F)} f(v) = 3\card{B} + \card{\D} + \card{A_2} + \card{A_3}$. 
By Lemma \ref{ConfigurationTypeOneEuler}, for each $y \in B$ we have $d_F(y) \ge k-1$.  
Hence $\size{F} \ge (k-1)\card{B}$.  This gives $(k-4)\card{B} \le \card{\D} + \card{A_2} + \card{A_3}$.  
Now the first inequality in the lemma follows since $B = V(\HH^-)$, $\card{A_3} = \sum_{v \in V(\HH^-)} d_{\HH}(v)$ and
\begin{align*}
\card{A_2} &= -q(\L) + \size{\HH, \L} \\
&= -q(\L) + k\card{\HH^-}+ \size{\HH^+, \L} - \sum_{v \in V(\HH^-)} d_{\HH}(v).
\end{align*}

Suppose $k \ge 7$.  Define $\func{f}{V(F)}{\IN}$ by $f(v) = 1$ for all $v \in A_2 \cup A_3$ and $f(v) = 2$ for all $v \in B \cup A_1$.  First, suppose $\size{F} > \sum_{v \in V(F)} f(v)$.  
Then by Lemma \ref{DegenerateEuler}, $F$ has an induced subgraph $Q$ such that $d_Q(v) > f(v)$ for each $v \in V(Q)$.  
In particular, $V(Q) \subseteq B \cup A_1$ and $\delta(Q) \ge 3$.  Put $Y \DefinedAs B \cap V(Q)$ and let $X$ be $\bigcup_{T \in V(Q) \cap A_1} V(T)$. 
Now $Z \DefinedAs G[X \cup Y]$ satisfies the hypotheses of Lemma \ref{MultipleHighConfigurationEuler}, so $Z$ has an induced subgraph $G'$ that is $f$-AT 
where $f(y) = d_{G'}(y) - 1$ for $y \in Y$ and $f(v) = d_{G'}(v)$ for $v \in X$.  Since $Y \subseteq B$ and $X \subseteq V(\L)$, we have $f(v) = k-1 + d_{G'}(v) - d_G(v)$ for all $v \in V(G')$.  
Hence, $G$ is AT-reducible to $G'$, a contradiction.

Therefore $\size{F} \leq \sum_{v \in V(F)} f(v) = 2(\card{B} + \card{\D}) + \card{A_2} + \card{A_3}$. 
By Lemma \ref{ConfigurationTypeOneEuler}, for each $y \in B$ we have $d_F(y) \ge k-1$.  
Hence $\size{F} \ge (k-1)\card{B}$.  This gives $(k-3)\card{B} \leq 2\card{\D} + \card{A_2} + \card{A_3}$.  
Now the second inequality in the lemma follows as before.
\end{proof}

\section{Bounding $\beta(\L)$}
This section is devoted to extracting the reusable Lemma \ref{betaLemma} from the proof of R. \cite{Better4ListCriticalBound}.

\begin{defn} A graph $G$ is \emph{OC-reducible} to $H$ if $H$ is a nonempty induced
subgraph of $G$ which is online $f_{H}$-choosable where $f_{H}(v)\DefinedAs\delta(G)+d_{H}(v)-d_{G}(v)$
for all $v\in V(H)$. If $G$ is not OC-reducible to any nonempty induced subgraph,
then it is \emph{OC-irreducible}. 
\end{defn}

\begin{lem}\label{betaLemma}
	Let $G$ be an OC-irreducible graph with $\delta(G) = k-1$.  Let $\L$ be the subgraph of $G$ induced on $(k-1)$-vertices and	$\HH$ the subgraph of $G$ induced on $k^+$-vertices.  
	If $2 \le \lambda \le \frac{6(k-1)}{k}$, then
	\[\beta(\L) \le \frac{2}{\lambda}\size{\HH} + \frac{2\size{G} - (k-2)\card{G} - \parens{\frac{k}{2} + \frac{k-1}{\lambda}}\card{\HH} - 1}{k-1}.\]
\end{lem}

\begin{observation}
The hypotheses of Lemma \ref{betaLemma} are satisfied by $k$-critical, $k$-list-critical and online $k$-list-critical graphs.
\end{observation}

The proof of Lemma \ref{betaLemma} requires the following lemma from Kierstead and R. \cite{KernelMagic} 
that generalizes a kernel technique of Kostochka and Yancey \cite{kostochkayancey2012ore}.

\begin{TheDefinition} The \emph{maximum independent cover number }of a graph $G$
	is the maximum $\mic(G)$ of $\size{I, V(G) \setminus I}$ over all independent sets $I$
	of $G$. 
\end{TheDefinition}

\begin{KernelMagic}\label{ConsantListMicStrength} 
	Every OC-irreducible graph $G$ with $\delta(G) = k-1$ satisfies
	\[2\size{G} \ge (k-2)\card{G} + \mic(G) + 1.\]
\end{KernelMagic}

\begin{thm}[L{\"o}wenstein, et al. \cite{lowenstein2011independence}]\label{Lowenstein}
If $G$ is a connected graph, then
\[\alpha(G) \ge \frac23\card{G} - \frac14\size{G} - \frac13.\]
\end{thm}

\begin{cor}\label{AlphaBound}
If $G$ is a connected graph, then
\[\alpha(G) \ge \frac23\card{G} - \frac13\size{G}.\]
\end{cor}
\begin{proof}
By Theorem \ref{Lowenstein},
\[\alpha(G) \ge \frac23\card{G} - \frac13\size{G} + \frac{1}{12}\size{G} - \frac13,\]
so, the corollary holds if $\frac{1}{12}\size{G} \ge \frac13$.  If not, then $\size{G} < 4$, so $G$ is
$K_1$, $K_2$, $P_3$ or $K_3$ which all satisfy the desired bound.
\end{proof}

\begin{proof}[Proof of Lemma \ref{betaLemma}]
Fix $\lambda$ with $2 \le \lambda \le \frac{6(k-1)}{k}$. Let $M$ be the maximum of $\size{I, V(G) \setminus I}$ over all independent sets $I$ of $G$ with $I \subseteq \HH$.   Since the vertices in $\L$ with $k-1$ neighbors in $\L$ have no neighbors in $\HH$,
	\begin{equation}
		\mic(G) \ge M + (k-1)\beta(\L).\label{eq0}
	\end{equation}
\claim{1}{If $C$ is a component of $\HH$, then \[k\alpha(C) \ge \parens{\frac{k}{2} + \frac{k-1}{\lambda}}\card{C} - \parens{\frac{2(k-1)}{\lambda}}\size{C}.\]}

First, suppose $\size{C} < \card{C}$.  Then $\size{C} = \card{C} - 1$ and $C$ is a tree.  If $\card{C} \ge 2$, then
	\begin{align*}
		k\alpha(C) &\ge k\frac{\card{C}}{2}\\
		&\ge \parens{\frac{k}{2} - \frac{k-1}{\lambda}}\card{C} + \frac{2(k-1)}{\lambda}\\
		&=\parens{\frac{k}{2} + \frac{k-1}{\lambda}}\card{C} - \parens{\frac{2(k-1)}{\lambda}}\parens{\card{C} - 1}\\
		&=\parens{\frac{k}{2} + \frac{k-1}{\lambda}}\card{C} - \parens{\frac{2(k-1)}{\lambda}}\size{C}.
	\end{align*}
If instead, $\card{C} = 1$, then $k\alpha(C) = k \ge \parens{\frac{k}{2} + \frac{k-1}{\lambda}} = \parens{\frac{k}{2} + \frac{k-1}{\lambda}}\card{C} - \parens{\frac{2(k-1)}{\lambda}}\size{C}$ since $\lambda \ge 2$.

So, we may assume $\size{C} \ge \card{C}$.  Applying Corollary \ref{AlphaBound}, we conclude
	\begin{align*}
		k\alpha(C) &\ge \frac{2k}{3}\card{C} - \frac{k}{3}\size{C} \\
		&= \parens{\frac{k}{2} + \frac{k-1}{\lambda}}\card{C} - \parens{\frac{2(k-1)}{\lambda}}\size{C} +  \parens{\frac{k}{6} - \frac{k-1}{\lambda}}\card{C} - \parens{\frac{k}{3} - \frac{2(k-1)}{\lambda}}\size{C} \\
		&= \parens{\frac{k}{2} + \frac{k-1}{\lambda}}\card{C} - \parens{\frac{2(k-1)}{\lambda}}\size{C} +  \parens{\frac{k-1}{\lambda} -\frac{k}{6}}\card{C}\\
		&\ge \parens{\frac{k}{2} + \frac{k-1}{\lambda}}\card{C} - \parens{\frac{2(k-1)}{\lambda}}\size{C},
	\end{align*}
	where in the final inequality we used $\lambda \le \frac{6(k-1)}{k}$.
	
\smallskip

\noindent\claim{2}{Lemma \ref{betaLemma} is true.}

	Summing the bound in Claim 1 over all components of $\HH$ and plugging into \eqref{eq0} gives
	\begin{equation}
	 \mic(G) \ge \parens{\frac{k}{2} + \frac{k-1}{\lambda}}\card{\HH} - \parens{\frac{2(k-1)}{\lambda}}\size{\HH} + (k-1)\beta(\L).
	 \label{eq19}
	\end{equation}
	Applying Kernel Magic using \eqref{eq19} and solving for $\beta(\L)$ proves the claim.
\end{proof}

\bibliographystyle{amsplain}
\bibliography{GraphColoring1}
\end{document}